\documentclass[12pt]{amsart}

\usepackage{lmodern}

\usepackage[T1]{fontenc}
\usepackage[utf8]{inputenc}
\usepackage[english]{babel}
\usepackage{amssymb,amsfonts,amsthm,amsmath,latexsym}
\usepackage{enumerate, caption, array, esint, color}
\usepackage{hyperref, fullpage}

\theoremstyle{plain}
\newtheorem{theorem}{Theorem}[section]

\newtheorem{lemma}[theorem]{Lemma}
\newtheorem*{lemma*}{Lemma}

\theoremstyle{definition}

\theoremstyle{remark}
\newtheorem*{remark}{Remark}

\newcommand\numberthis{\stepcounter{equation}\tag{\theequation}}

\newcommand{\ssum}[1]{\sum_{\substack{#1}}}
\newcommand{\e}{{\rm e}}
\newcommand{\dd}{{\rm d}}
\newcommand{\eps}{{\varepsilon}}

\newcommand{\R}{{\mathbb R}}
\newcommand{\N}{{\mathbb N}}
\newcommand{\Z}{{\mathbb Z}}
\newcommand{\1}{{\mathbf 1}}

\newcommand{\cR}{{\mathcal R}}

\DeclareMathOperator{\cond}{cond}

\newcommand{\vphi}{{\varphi}}

\newcommand{\chit}{{\tilde \chi}}
\newcommand{\qt}{{\tilde q}}
\newcommand{\mut}{{\tilde \mu}}

\renewcommand{\mod}[1]{\ ({\rm mod\ }#1)}
\renewcommand\Re{\operatorname{\mathfrak{Re}}}
\renewcommand\Im{\operatorname{\mathfrak{Im}}}

\numberwithin{equation}{section}

\title[First moment of primes in arithmetic progressions]
{The first moment of primes in arithmetic progressions: Beyond the Siegel-Walfisz range}

\author{Sary Drappeau}
\address{Aix Marseille Université, CNRS, Centrale Marseille, I2M UMR 7373, 13453, Marseille, France}
\email{sary-aurelien.drappeau@univ-amu.fr}

\author{Daniel Fiorilli}
\address{CNRS, Université Paris-Saclay, Laboratoire de mathématiques d'Orsay, 91405, Orsay, France.}
\address{D\'epartement de math\'ematiques et de statistique, Universit\'e d'Ottawa, 585 King Edward, Ottawa, Ontario, K1N 6N5, Canada}
\email{daniel.fiorilli@universite-paris-saclay.fr}

\date{\today}


\begin{document}

\begin{abstract}
  We investigate the first moment of primes in progressions
  $$ \ssum{q\leq x/N \\ (q, a)=1}  \Big(\psi(x; q, a) - \frac x{\vphi(q)}\Big) $$
  as~$x, N \to \infty$. We show unconditionally that, when~$a=1$, there is a significant bias towards negative values, uniformly for $N\leq \e^{c\sqrt{\log x}}$. The proof combines recent results of the authors on the first moment and on the error term in the dispersion method. 
More generally, for $a \in \mathbb Z\setminus\{0\}$ we prove estimates that take into account the potential existence (or inexistence) of Landau-Siegel zeros.
\end{abstract}

\maketitle

\section{Introduction}

The distribution of primes in arithmetic progressions is a widely studied topic, in part due to its links with binary additive problems involving primes, see \emph{e.g.}~\cite[Chapter~19]{IwaniecKowalski2004} and~\cite{Linnik1963a}. For all~$n\in\N$ we let~$\Lambda$ denote the von Mangoldt function,
and for a modulus~$q\in\N$ and a residue class~$a\mod{q}$ we define
$$ \psi(x;q, a) := \ssum{n\leq x \\ n\equiv a\mod{q}} \Lambda(n). $$
In the work~\cite{Fiorilli2012a}, the second author showed the existence, for certain residue classes~$a$, of an unexpected bias in the distribution of primes in large arithmetic progressions, on average over~$q$. An important ingredient in this result is the dispersion estimates of Fouvry~\cite{Fouvry1985} and Bombieri-Friedlander-Iwaniec~\cite{BombieriFriedlanderEtAl1986}; these involve an error term which restricts the range of validity of~\cite[Theorem 1.1]{Fiorilli2012a}. Recently, this error term was refined by the first author in~\cite{Drappeau2017}, taking into account the influence of potential Landau-Siegel zeros. This new estimate allows for an extension of the range of validity of~\cite[Theorem 1.1]{Fiorilli2012a}, which is the object of the present paper. 
In particular, we quantify and study the influence of possible 
Landau-Siegel zeros, and we show that, in the case~$a=1$, a bias subsists \emph{unconditionally} in a large range. Here is our main result.

\begin{theorem}\label{thm:maina=1}
There exists an absolute constant $\delta >0$ such that for any fixed $\eps>0$ and in the range~$1\leq N\leq \e^{\delta\sqrt{\log x}}$, we have the upper bound
  
\begin{equation}
   \frac Nx\ssum{q\leq x/N }  \Big(\psi(x; q, 1) - \frac x{\vphi(q)}\Big) \leq -\frac {\log N}2 -C_0 + O_\eps(N^{-\frac{171}{448}+\eps}),
   \label{equation:maina=1}
\end{equation}  
 with an implicit constant depending effectively on~$\eps$, and where 
\begin{equation*}
  C_0:= \frac 12 \left(\log 2\pi + \gamma +\sum_p \frac {\log p}{p(p-1)}+1 \right).
\end{equation*}
  
\end{theorem}
In other words, there is typically a negative bias towards the class $a=1$ in the distribution of primes in arithmetic progressions modulo $q$.
One could ask whether Theorem~\ref{thm:maina=1} could be turned into an asymptotic estimate. To do so we would need to rule out the existence of Landau-Siegel zeros, because if they do exist, then we find in Theorem~\ref{theorem main} below that the left hand side of~\eqref{equation:maina=1} is actually much more negative. 

In order to explain our more general result, we will need to introduce some notations and make a precise definition of Landau-Siegel zeros. We begin by recalling~\cite[Theorem 1.1]{Fiorilli2012a}.
For $N\geq 1$ and $a\in \Z\setminus \{0\}$ we define
$$ M_1(x, N ; a) = \ssum{q\leq x/N \\ (q, a)=1}  \Big(\psi^*(x; q, a) - \frac x{\vphi(q)}\Big), $$
where\footnote{Note that we have excluded the first term because it has a significant contribution which is trivial to estimate.} $$\psi^*(x;q,a) = \ssum{ 1\leq n\leq x \\ n\equiv a \bmod q \\ n\neq a} \Lambda(n). $$
With these notations, \cite[Theorem 1.1]{Fiorilli2012a} states\footnote{The improved exponent is deduced by applying Bourgain's work \cite{Bourgain2017}.} that for~$N\leq (\log x)^{O(1)}$ 
\begin{equation}\label{eq:result-DF}
  \frac {M_1(x,N;a)}{\frac{\phi(|a|)}{|a|} \frac xN} = \mu(a,N)+O_{a,\eps,B} \left(N^{-\frac {171}{448}+\eps}\right)
\end{equation}
with 
\begin{equation}
  \mu(a,N):=  \begin{cases}
    -\frac 12 \log N - C_0 &\text{ if } a=\pm 1\\
    -\frac 12 \log p &\text{ if } a=\pm p^e \\
    0 &\text{ otherwise.}
  \end{cases}
  \label{equation definition of mu}
\end{equation}

We recall the following classical theorem of Page.

\begin{theorem}[{\cite[Theorems~5.26 and~5.28]{IwaniecKowalski2004}}]
  There is an absolute constant~$b>0$ such that for all~$Q, T \geq 2$, the following holds true. The function~$s\mapsto \prod_{q\leq Q} \prod_{\chi\mod{q}} L(s, \chi)$ has at most one zero~$s=\beta$ satisfying~$\Re(s)>1-b/\log(QT)$ and~$\Im(s)\leq T$. If it exists, the zero~$\beta$ is real and it is the zero of a unique function~$L(s, \chit)$ for some primitive real character~$\chit$.
\end{theorem}

Given~$x\geq 2$, we will say that the character~$\chit\mod{\qt}$ is~$x$-exceptional if the above conditions are met with~$Q=T=\e^{\sqrt{\log x}}$. There is at most one such character.

By the analytic properties of Dirichlet $L$-functions, if exceptional zeros exist, their effect can often be quantified in a precise way, and are expected to lead to secondary terms in asymptotic formulas. For instance, it is known~\cite[Corollary~11.17]{MontgomeryVaughan2006} that if the~$x$-exceptional character exists, then there is a distortion in the distribution of primes in the sense that
\begin{align*}
  \psi(x ; q, a) ={}& \frac{x}{\vphi(q)} - \chit(a)\1_{\qt \mid q} \frac{x^\beta}{\beta \vphi(q)} + O(x\e^{-c\sqrt{\log x}})  \numberthis\label{eq:psi-SW-2terms}\\
  ={}& \frac{x}{\vphi(q)}(1-\eta_{x,a} \1_{\qt\mid q}) + O(x\e^{-c\sqrt{\log x}}) \numberthis\label{eq:psi-SW}
\end{align*}
with
\begin{equation}
 \eta_{x,a} := \frac{\chit(a)}{\beta x^{1-\beta}} \in (-1, 1). 
 \label{equation definition eta}
\end{equation}

We are now ready to state our more general result. As we will see, the secondary term in~\eqref{eq:psi-SW-2terms} can potentially yield a large contribution to $M_1(x, N; a)$ for~$N$ considerably larger than~$\qt$. For this reason, it is relevant to consider instead the expression
\begin{equation}
  M_1^Z(x, N ; a) = \ssum{q\leq x/N \\ (q, a)=1}  \Big(\psi^*(x; q, a) - (1-\1_{\qt\mid q}\eta_{x,a}) \frac x{\vphi(q)} \Big),\label{eq:def-M1Z}
\end{equation}
where, by convention, the term involving~$\eta_{x,a}$ is only to be taken into account when the~$x$-exceptional character exists.

Our results show that, in the case of the hypothetical two-term approximation~\eqref{eq:def-M1Z}, there is a new bias term, which results from the contribution of the possible $x$-exceptional character.

\begin{theorem}
  \label{theorem main}
  Fix an integer~$a\in \Z\setminus \{0\}$ and a small enough positive absolute constant $\delta$, and let~$x\geq 2$ and~$2\leq N\leq \e^{\delta\sqrt{\log x}}$.
  \begin{enumerate}[(i)]
    \item If there is no~$x$-exceptional character, then
    \begin{equation}\label{eq:main-estim-1}
      \frac {M_1(x,N;a)}{\frac{\phi(|a|)}{|a|} \frac xN} = \mu(a,N)+O_{a,\eps} \left(N^{-\frac {171}{448}+\eps}\right).
    \end{equation}
    \item If the~$x$-exceptional character~$\chit\mod{\qt}$ exists, then with~$C_{a,\qt}$ and~$D_{a,\qt}$ as in~\eqref{eq:def-Caq} and~\eqref{eq:def-Daq} below,
    \begin{equation}\label{eq:main-estim-2}
      \begin{aligned}
        \frac{M_1^Z(x, N; a)}{\frac{\vphi(|a|)}{|a|} \frac xN} = \mu(a, N) + {}& N \eta_{x,a} \Big(\ssum{r\leq N \\ (r, a)=1 \\ \qt \mid r} \frac{1 - (\frac rN)^\beta}{\vphi(r)} - C_{a,\qt}\Big\{\log\Big(\frac N\qt\Big) + D_{a,\qt} - \frac1\beta\Big\}\Big) \\
        {}& + O_{a,\eps} \left(N^{-\frac {171}{448}+\eps}\right).
      \end{aligned}
    \end{equation}
    \item If the~$x$-exceptional character exists and~$N\geq \qt$, then the previous formula admits the approximation
    \begin{equation}\label{eq:main-estim-3}
      \frac{M_1^Z(x, N; a)}{\frac{\vphi(|a|)}{|a|} \frac xN} = (1-\eta_{x,a}) \mu(a, N) +  \eta_{x,a} \mut_\chit(a) + O_{a,\eps}\Big(N^\eps (N/\qt)^{-\frac{171}{448}}+ (\log N)^2\frac{1-\beta}{x^{(1-\beta)/2}}\Big),
    \end{equation}
    where~$\mut_\chit(a) = \tfrac12 \1_{a=\pm 1}(\log\qt - \sum_{p\mid\qt} \frac{\log p}{p})$.
  \end{enumerate}
\end{theorem}

In~\eqref{eq:main-estim-2}, we have that~$C_{a,\qt}\ll_a 1/\phi(\qt)$ and~$D_{a,\qt} \ll_a 1$, hence the secondary term involving~$\eta_{x,a}$ is~$O_{a,\eps}(N(\log N)\qt^{-1+\eps})$. Since by Siegel's theorem we have the bound~$\qt \gg_A (\log x)^A$ for any fixed $A>0$, we recover~\cite[Theorem 1.1]{Fiorilli2012a}.

Finally we remark that if the $x$-exceptional character exists and~$N\geq \qt$, the associated ``secondary bias'', that is the difference between the main terms on the right hand side of~\eqref{eq:main-estim-3} and $\mu(a,M)$, contributes an additional quantity
$$ - \eta_{x,a}(\mu(a, N) - \mut_{\chit}(a)). $$
The bound~$\qt \gg_A (\log x)^A$
does not exclude the possibility that~$(1-\beta)\log x = o(1)$ in the context of~\eqref{eq:main-estim-3}. Should this happen, we would have that~$\eta_{x,a} = \chit(a) + o(1)$. If moreover $a=1$ and~$N\leq \qt^{O(1)}$, then the main term of~\eqref{eq:main-estim-3} would becomes asymptotically~$(1+o(1))\mut_\chit(a)$, and would not depend on~$N$ anymore. In this situation, the additional bias coming from the exceptional character would annihilate the~$N$-dependance of the overall bias.

\begin{remark}
  The influence of possible Landau-Siegel zeros on the second moment has been investigated by Liu in~\cite{Liu08}. As for the first moment, it is closely related to the Titchmarsh divisor problem of estimating, as~$x\to \infty$, the quantity
  $$ \sum_{1< n\leq x} \Lambda(n) \tau(n-1). $$
  After initial works of Titchmarsh~\cite{Titchmarsh1930} and Linnik~\cite{Linnik1963a}, Fouvry \cite{Fouvry1985} and Bombieri, Friedlaner and Iwaniec \cite{BombieriFriedlanderEtAl1986} were able to show a full asymptotic expansion, with an error term~$O(x/(\log x)^A)$.
  In the recent work~\cite{Drappeau2017}, the first author refined this estimate taking into account the influence of possible Landau-Siegel zero, with an error term~$O(\e^{-c\sqrt{\log x}})$.
\end{remark}

\section{Proof of Theorem \ref{theorem main}}

\subsection{The Bombieri-Vinogradov range}

We begin with the following lemma, which follows from the large sieve and the Vinogradov bilinear sums method. Given a Dirichlet character $\chi \bmod q$, we let
$$ \psi(x, \chi) := \sum_{n\leq x} \chi(n) \Lambda(n). $$
\begin{lemma}
  For~$2\leq R \leq Q \leq \sqrt{x}$, we have the bound
  $$ \sum_{q\leq Q} \frac1{\vphi(q)} \ssum{\chi \mod{q} \\ R < \cond(\chi) \leq Q}\max_{y\leq x} |\psi(y, \chi)| \ll (\log x)^{O(1)}\big\{R^{-1}x + Q\sqrt{x} + x^{\frac 56}\big\}. $$
\end{lemma}
\begin{proof}
  This follows from the third display equation of page 164 of~\cite{Davenport} with~$Q_1 = R$, after reintegrating imprimitive characters as in~\cite[page~163]{Davenport}.
\end{proof}

We deduce the following version of the Bombieri-Vinogradov theorem, with the contribution of exceptional zeros removed.
\begin{lemma}\label{lem:BV-sqrtlog}
Fix $a\in \Z\setminus \{0\}$.  There exists~$\delta>0$ such that for all~$x, Q\geq 1$ we have the bound
  $$ \sum_{q\leq Q} \max_{y\leq x} \max_{(a, q)=1} \Big|\psi(y; q, a) - (1-\eta_{x,a}\1_{\qt\mid q})\frac {y}{\vphi(q)} \Big| \ll x\e^{-\delta\sqrt{\log x}} + Q\sqrt{x}(\log x)^{O(1)}, $$
  where~$\eta_{x,a}$ was defined in~\eqref{equation definition eta}.
\end{lemma}

\subsection{Initial transformations, divisor switching}

From now on, we let~$\delta>0$ be a positive parameter that will be chosen later, we fix $x \geq 1$ and define~$\chit$ as the possible $x$-exceptional character, of conductor~$\qt$ and associated zero~$\beta$, and recall the notation~\eqref{equation definition of mu}. In order to isolate the contribution of the potential Landau-Siegel zero, we define
$$ E(x; q, a) := \psi^*(x; q, a) - (1-\eta_{x,a}\1_{\qt\mid q})\frac x{\vphi(q)}. $$
If the $x$-exceptional character does not exist, then every term involving $\eta_{x,a}$ can be deleted. With this notation we have the decomposition
\begin{align}
  \notag 
  M_1^Z(x, N ; a) = {} & \ssum{q \leq x^{\frac12+\delta} \\ (q, a)=1} E(x; q, a) + \ssum{x^{\frac12+\delta}<q \leq x \\ (q,a)=1} \psi^*(x; q, a) - \ssum{x/N <q \leq x \\ (q,a)=1} \psi^*(x; q, a)  \\ 
  {}& \notag \qquad - x \ssum{x^{\frac12+\delta}<q\leq x/N \\ (q,a)=1} \frac1{\vphi(q)} + \eta_{x,a} x\ssum{x^{\frac12+\delta} < q \leq x/N \\ (q, a)=1 \\ \qt\mid q} \frac1{\vphi(q)}\\
  = {}& T_1 + T_2 - T_3 - T_4 - T_5,
  \label{equation initial decomposition}
\end{align}
say.
 We discard the first term by using the dispersion estimate~\cite[Theorem 6.2]{Drappeau2017}, which yields that there exists an absolte constant~$\delta>0$ such that in the range $|a|\leq x^{\delta}$,
\begin{equation}
 T_1= \ssum{q \leq x^{\frac12+\delta} \\ (q, a)=1} E(x; q, a) \ll x \e^{-\delta\sqrt{\log x}}. 
 \label{equation bound T1}
\end{equation}

We end this section by applying divisor switching to the sums $T_2$ and $T_3$.
\begin{lemma}
\label{lemma T2 T3}
Fix $a\in \Z\setminus \{0\}$ and define $T_2$ and $T_3$ as in~\eqref{equation initial decomposition}. There exists an absolute constant $\delta>0$ such that in the range $N\leq x^{\frac 12-\delta}$, $|a| N\leq x\e^{-2\delta \sqrt{\log x}}$ we have the estimate
\begin{multline*}
 T_2-T_3 =  x\ssum{r\leq x^{\frac12-\delta} \\ (r,a)=1} \frac{1-(\frac r{x^{1/2-\delta}})}{\vphi(r)} -  \eta_{x, a} x\ssum{r\leq x^{\frac12-\delta} \\ (r,a)=1 \\ \qt\mid  r} \frac{1 - (\tfrac r{x^{1/2-\delta}})^\beta}{\vphi(r)} \\
 -x\ssum{r\leq N \\ (r,a)=1} \frac{1-(\frac r N)}{\vphi(r)} + \eta_{x, a} x \ssum{r\leq N \\ (r,a)=1 \\ \qt\mid  r} \frac{1 - (\tfrac r N)^\beta}{\vphi(r)} + O(x\e^{-\delta\sqrt{\log x}}).
\end{multline*}
\end{lemma}

\begin{proof}

We rewrite the condition $n\equiv a \bmod q$ as~$n=a+qr$ for~$r\in\Z$. Summing over~$r$ and keeping in mind that $|a| N<x$, for large enough values of $x$ we obtain the formula
\begin{equation} \notag
  T_3 {} = \ssum{1\leq r < N-aN/x \\ (r,a)=1} \big( \psi^*(x; r, a) - \psi^*(a+\tfrac{rx}N ; r, a)\big).
\end{equation}

Recalling that~$N\leq x^{\frac12-\delta}$, we may apply the Bombieri-Vinogradov theorem in the form of Lemma~\ref{lem:BV-sqrtlog}. We obtain the estimate
\begin{equation}
  T_3 {}  = x\ssum{r\leq N \\ (r,a)=1} \frac{1-(\frac r N)}{\vphi(r)} - \eta_{x, a} x \ssum{r\leq N \\ (r,a)=1 \\ \qt\mid  r} \frac{1 - (\tfrac r N)^\beta}{\vphi(r)} + O(x\e^{-\delta\sqrt{\log x}}).
  \label{equation estimate for T_4}
\end{equation}
Replacing~$N$ by~$x^{\frac12-\delta}$, we obtain a similar estimate for~$T_2$, and the result follows.
\end{proof}

\label{section switch}

\subsection{Sums of multiplicative functions}

In the following sections, we collect the main terms obtained in the previous section and show that they cancel, to some extent, with $T_4$ and~$T_5$. We start with the following estimate for the mean value of~$1/\vphi(q)$, which we borrow from~\cite[Lemma~5.2]{Fiorilli2012} with the main terms identified in~\cite[Lemme~6]{Fouvry1985}.

\begin{lemma}
  \label{lemma Fouvry}
Fix $\eps>0$.  For $a \in \Z\setminus\{ 0\}$ and $q_0\in \mathbb N$ such that $(a,q_0)=1$ and $q_0\leq Q$, we have the estimate
  $$ \ssum{q\leq Q \\ (q, a)=1 \\ q_0\mid  q} \frac1{\vphi(q)} = C_{a,q_0}\Big\{\log\Big(\frac Q{q_0}\Big) + D_{a,q_0}\Big\} + O_{a, \eps}(q_0^\eps Q^{-1+\eps}), $$
  where
  \begin{equation}\label{eq:def-Caq}
    C_{a,q_0} := \frac{\phi(a)}{a\vphi(q_0)} \prod_{p\nmid aq_0} \Big( 1+ \frac 1{p(p-1)} \Big),
  \end{equation}
  \begin{equation}\label{eq:def-Daq}
    D_{a,q_0} := \sum_{p\mid a} \frac{\log p}{p} - \sum_{p\nmid aq_0} \frac{\log p}{p^2-p+1}+\gamma_0.
  \end{equation}
  Here,~$\gamma_0$ is the Euler-Mascheroni constant.
\end{lemma}

We now estimate the main terms in Lemma~\ref{lemma T2 T3}. For $N\in \N$, $a\in \Z\setminus \{0\}$ and $q_0 \in \N$ we define
$$ J_\gamma(x, N ; q_0, a) := \ssum{r\leq N \\ (r, a)=1 \\ q_0 \mid r} \frac{1 - (\frac rN)^\gamma}{\vphi(r)} + \ssum{q\leq x/N \\ (q, a)=1 \\ q_0 \mid q} \frac1{\vphi(q)}. $$

\begin{lemma}
  \label{lemma shift to 1/2}
  Fix~$\delta>0$ small enough and $a \in \Z\setminus\{ 0\}$. For $\gamma\in[\frac34, 1]$, $(q_0, a)=1$ and in the range~$1 \leq N \leq x^{1-\delta}$,~$1\leq q_0 \leq x^\delta$, we have the estimate
  \begin{equation}
    J_\gamma(x, N ; q_0, a) = {\tilde J}_\gamma(x ; q_0, a) + \frac{\gamma f_{q_0,a;N}(1)-f_{q_0,a;N}(\gamma)}{1-\gamma} + O_{a,\varepsilon}\big(N x^{-1+\eps} + q_0^{\frac{171}{448}+\varepsilon}N^{-1-\frac{171}{448}+\varepsilon} \big),\label{equation lemma shift to 1/2}
  \end{equation}
  where the implied constant does not depend on $\gamma$, the value of the second main term at $\gamma=1$ is defined by taking a limit, and
  $$ {\tilde J}_\gamma(x ; q_0, a) := C_{a,q_0}\Big\{\log\Big(\frac x{q_0^2}\Big) + 2 D_{a,q_0} - \frac1{\gamma}\Big\}; $$
  \begin{equation}
    f_{q_0, a;N}(\gamma) := \frac{(q_0/N)^\gamma}{\vphi(q_0)}Z(-\gamma)G_{q_0,a}(-\gamma)\zeta(1-\gamma)\zeta(2-\gamma)(1-\gamma),\label{eq:def-f}
  \end{equation}
 where
  $$Z(s) := \prod_p \Big( 1+\frac 1{p^{s+2}(p-1)}-\frac 1{p^{2s+3}(p-1)} \Big);$$
  $$ G_{q_0,a}(s) := \prod_{p\mid aq_0} 
  \Big(1+\frac 1{p^{s+1}(p-1)}\Big)^{-1} \prod_{p\mid a } \Big( 1-\frac 1{p^{s+1}} \Big). $$
\end{lemma}
\begin{proof}
  Assume that~$\gamma<1$. We will obtain error terms that are uniform in~$\gamma$; this will allow us to take a limit and the result with $\gamma=1$ will follow. We split Mellin inversion and a straightforward calculation gives the identity
  \begin{multline*}
    J_\gamma(x, N; q_0, a) = \frac1{2\pi i}\int_{(2)} \frac{1}{q_0^s\vphi(q_0)} Z(s) G_{q_0, a}(s) \zeta(s+1)\zeta(s+2)\Big\{ \frac{\gamma N^s}{s+\gamma} + \Big(\frac x N\Big)^s \Big\} \frac{\dd s}{s}.
  \end{multline*}
  Taking Taylor series shows that for $R\in \mathbb R_{\geq 1}$,
  $$ \frac{\gamma R^s}{s+\gamma} + \Big(\frac x R\Big)^s = 2 + s\Big(\log x - \frac1\gamma\Big) + O_{x,\gamma,R}(|s|^2) $$
  in a neighborhood of~$0$. We first shift the contour to the left until~$(-\tfrac12)$. The residue at~$s=0$ contributes exactly~${\tilde J}(x; q_0, a)$. We handle the contribution of the term~$(x/N)^s$ similarly as in \cite[Lemma 5.12]{Fiorilli2013}: a trivial estimation using a truncated Perron's formula shows that it is~$O_{a, \eps}((x/N)^{-1/2+\eps})$, while shifting back the contour to~$2+i\R$ (picking up a residue at~$s=0$) and applying Mellin inversion, we get
  $$ \frac1{2\pi i}\int_{(-\frac12)} \frac{1}{q_0^s\vphi(q_0)} Z(s) G_{q_0, a}(s) \zeta(s+1)\zeta(s+2) \Big(\frac x N\Big)^s \frac{\dd s}{s} = \ssum{q\leq x/N \\ (q, a)=1 \\ q_0 \mid q} \frac1{\vphi(q)} - K_1\log\Big(\frac x{q_0N}\Big) - K_2 $$
  for some constants~$K_1, K_2$ depending on~$q_0$ and~$a$. Applying Lemma~\ref{lemma Fouvry}, we identify~$K_1 = C_{a,q_0}$ and~$K_2 = D_{a,q_0}$, and we deduce that the above is~$O_{a,\eps}((x/N)^{-1+\eps})$. We deduce that
  \begin{align*}
    J_\gamma(x, N; q_0, a) = {}& {\tilde J}(x; q_0, a) + O_{a,\eps}\Big(\Big(\frac{N}{x}\Big)^{1-\eps}\Big)\\
    {}&\qquad + \frac1{2\pi i}\int_{(-\frac12)} \frac{1}{q_0^s\vphi(q_0)} Z(s) G_{q_0, a}(s) \zeta(s+1)\zeta(s+2) \frac{\gamma N^s}{s+\gamma} \frac{\dd s}{s}.
  \end{align*}
  Shifting the remaining integral further to the line~$(-1-\eps)$, we pick up two residues, at~$s=-1$ and at~$s=-\gamma$. This gives rise to the second term in \eqref{equation lemma shift to 1/2}. As for the shifted integral, we apply Bourgain's subconvexity estimate for $\zeta(s)$ \cite{Bourgain2017}. Note that 
  $$ G_{q_0,a}(s) \ll \prod_{p\mid aq_0} \Big(1+\frac 1{p^{\Re(s)+2}}\Big) \prod_{p\mid a} \Big(1+\frac 1{p^{\Re(s)+1}}\Big). $$
  As in \cite[Lemma 5.9]{Fiorilli2012a}, we shift the contour to the line $\Re(s) =-1-1/(2+4\theta)$, where $\theta= 13/81$ is Bourgain's subconvexity exponent. The shifted integral is
  $$ \ll_{a,\varepsilon} q_0^{1/(2+4\theta)+\varepsilon}N^{-1-1/(2+4\theta)+\varepsilon}. $$
  The desired estimate follows.
\end{proof}

In the next two sections, we will prove approximations for the term
\begin{equation}\label{eq:def-D}
  D_\gamma(q_0, a;N) :=  \frac{\gamma f_{q_0,a;N}(1)-f_{q_0,a;N}(\gamma)}{1-\gamma}
\end{equation}
appearing in~\eqref{equation lemma shift to 1/2}.

\subsection{The main term for~$\gamma=1$}

The limit of~$D_\gamma(q_0, a; N)$ as~$\gamma\to 1$ has a simple expression in terms of derivatives of~$f_{q, a, N}$, namely
$$ D_1(q_0, a ; N) = f_{q_0,a,N}'(1) - f_{q_0,a,N}(1). $$
Recall that~$f_{q_0,a,N}$ is given by the Euler product~\eqref{eq:def-f}. A direct computation yields that for $q\in \N, a\in \Z\setminus \{0\}$ and $N\in \R_{\geq 1}$,
\begin{align*}
  f_{q,a,N}(1) ={}& \begin{cases}0 & \text{if }a\neq \pm 1; \\ -\frac1{2N} & \text{ if } a= \pm 1, \end{cases} \\
  f'_{q,a,N}(1) ={}& \begin{cases}0 & \text{if } \omega(a)\geq 2; \\ \frac1{2N}(1-\frac1\ell)\log \ell & \text{if } a= \pm \ell^\nu \ (\nu\in\N, \ell\text{ prime}); \\ -\frac1{2N}\big\{\log(\frac qN) - 2C_0 + 1 - \sum_{p|q}\frac{\log p}{p}\big\} & \text{if }a=\pm 1. \end{cases}
\end{align*}

From these observations, we deduce the following.
\begin{lemma}\label{lem:approx-D-at1}
  We have the exact formula
  $$ D_1(q, a; N) = \begin{cases}0 & \text{if } \omega(a)\geq 2; \\ \frac1{2N}(1-\frac1\ell)\log \ell & \text{if } a=\pm \ell^\nu\ (\nu\in\N, \ell\text{ prime}); \\ \frac1{2N}\big\{\log(\frac Nq) + 2C_0 + \sum_{p|q}\frac{\log p}{p}\big\} & \text{if }a=\pm 1. \end{cases} $$
\end{lemma}

\subsection{The main term for~$\gamma<1$}
Now that we have estimated the main term in Lemma~\ref{lemma shift to 1/2} for $\gamma=1$, we will do so for $\gamma<1$.
Under this restriction, we write
$$ D_\gamma(q_0, a ; N) - D_1(q_0, a ; N) = \frac{1}{\gamma-1}\int_\gamma^1 \int_{\delta}^1 f_{q_0,a,N}''(\delta')\dd\delta' \dd\delta. $$
By a direct estimation of the Euler product we see that in the range $\tfrac 34\leq \gamma \leq \delta'\leq 1$,
$$ |f_{q_0,a,N}''(\delta')| 
\ll_a \frac{q_0^{\delta'}|G_{q_0,a}(-\delta')|}{\phi(q_0)} (\log q_0N)^2 N^{-\gamma}\ll (\log q_0N)^2 N^{-\gamma}. $$
Therefore, when~$\tfrac34\leq \gamma\leq 1$, we obtain
\begin{equation}
  D_\gamma(q_0, a ; N) = D_1(q_0, a ; N) + O_a( (\log q_0N)^2 (1-\gamma) N^{-\gamma} ).\label{eq:compar-D1-Dbeta}
\end{equation}
Along with Lemma~\ref{lem:approx-D-at1}, the above yields the following approximation.
\begin{lemma}\label{lem:approx-D-away1}
  Define
  $$ \cR(x,N) := \frac{ (1-\beta)(\log \qt N)^2}{N^\beta x^{1-\beta}}. $$
  For~$(a, \qt)=1$, $\nu\in\N$ and~$\ell$ prime, we have that
  \begin{align*}
    D_1(&1, a ; N) - \eta_{x,a}D_\beta(\qt, a ; N) \\
    {}& = \begin{cases}
      O_a(\cR(x,N)) & \text{if }\omega(a) \geq 2; \\
      (1-\eta_{\chi, a})\frac1{2N}(1 - \frac 1\ell)\log \ell + O_a(\cR(x,N)) & \text{if } a=\pm\ell^\nu; \\
      (1-\eta_{x,a})\frac1{2N}\{\log N + 2C_0\} + \eta_{x,a} \frac1{2N} \{\log \qt - \sum_{p|q}\frac{\log p}{p}\} + O_a(\cR(x,N)) & \text{if }a=\pm 1.
    \end{cases}
  \end{align*}

\end{lemma}

\subsection{Cancellation of main terms and proof of Theorem~\ref{theorem main}}

In this section we combine the main terms in $T_2,T_3, T_4$ and~$T_5$ and prove our main theorem. 

\begin{proof}[Proof of Theorem~\ref{theorem main}]

Recalling~\eqref{equation initial decomposition}, we have by~\eqref{equation bound T1} and Lemmas~\ref{lemma T2 T3} and~\ref{lemma shift to 1/2} that for some small enough $\delta>0$,
\begin{align*}
  {}&M_1^Z(x, N ; a) =T_1+ T_2 - T_3 - T_4 + T_5 \\
  = {}& x \big\{J_1(x, x^{\frac 12-\delta}, 1, a) - J_1(x, N, 1, a)\big\} - \eta_{x,a}x\big\{J_\beta(x, x^{\frac 12-\delta}, \qt, a) - J_\beta(x, N, \qt, a)\big\} + O(x\e^{-\delta\sqrt{\log x}}) \\
  = {}& -x D_1(1, a;N) - \eta_{x,a}x\big\{{\tilde J}_\beta(x, \qt, a) - J_\beta(x, N, \qt, a)\big\} + O_{a,\eps}(x  N^{-1-\frac{171}{448}+\eps}).
\end{align*}
Here, we used the bound~$D_1(q, a, N) \ll_a N^{-1}(\log qN)$ along with~\eqref{eq:compar-D1-Dbeta}.
If the~$x$-exceptional character does not exist, then this yields~\eqref{eq:main-estim-1}.

Next, assume that the~$x$-exceptional character does exist, and that~$N\ll \qt$. Then by definition and since $\widetilde q\leq e^{\sqrt {\log x}}$,
\begin{align*}
  {}& {\tilde J}_\beta(x, \qt, a) - J_\beta(x, N, \qt, a) \\
  = {}& C_{a,\qt}\Big\{\log\Big(\frac{x}{\qt^2}\Big) + 2D_{a,\qt} - \frac1\beta\Big\} - \ssum{r\leq N \\ \qt \mid r\\ (a, r)=1} \frac{1-(r/N)^\beta}{\vphi(r)} - \ssum{q\leq x/N \\ \qt \mid q \\ (q, a)=1} \frac1{\vphi(q)} \\
  = {}& C_{a,\qt}\Big\{\log\Big(\frac{N}{\qt}\Big) + D_{a,\qt} - \frac1\beta\Big\} - \ssum{r\leq N \\ \qt \mid r\\ (a, r)=1} \frac{1-(r/N)^\beta}{\vphi(r)} + O(x^{-\frac 15}),
\end{align*}
where the sum over~$q$ was evaluated using Lemma~\ref{lemma Fouvry}. Since~$N\leq \e^{\delta\sqrt{\log x}}$, this yields~\eqref{eq:main-estim-2}.

Assume now that the $x$-exceptional character exists and that~$N\geq \qt$. We use Lemma~\ref{lemma shift to 1/2} to write
$$ {\tilde J}_\beta(x;\qt,a) - J_\beta(x,N;\qt,a) = -D_\beta(\qt, a ; N) + O_{a,\eps}(N^{-1+\eps} (\qt / N)^{\frac{171}{448}}). $$
Therefore,
$$ M_1^Z(x, N ; a)= -x\Big\{D_1(1, a ; N) - \eta_{x,a} D_\beta(\qt, a ; N) + O_{a,\eps}(N^{-1+\eps} (\qt / N)^{\frac{171}{448}})\Big\}. $$
Our claimed formula~\eqref{eq:main-estim-3} then follows from Lemma~\ref{lem:approx-D-away1}. 
\end{proof}

\subsection{Unconditional bias}
In this last section we prove our unconditional result.

\begin{proof}[Proof of Theorem~\ref{thm:maina=1}]
If the~$x$-character does not exists, then the claimed bound follows from~\eqref{eq:main-estim-1}. We can therefore assume that it does exists. Note that
\begin{align*}
  \frac{M_1(x, N ; 1)}{x/N} = {}& \frac{M_1^Z(x, N; 1)}{x/N} - N\eta_{x,1}\ssum{q\leq x/N \\ \qt \mid q} \frac1{\vphi(q)} \\
  = {}& \frac{M_1^Z(x, N; 1)}{x/N} - N\eta_{x,1}C_{1,\qt}\Big\{\log\Big(\frac{x}{N\qt}\Big) + D_{1,\qt}\Big\} + O(x^{-\frac15}).
\end{align*}
Using our estimate~\eqref{eq:main-estim-2}, and noting that the~$r$-sum is~$O(\log (\widetilde qN)/\vphi(\qt))$, we obtain that
$$ \frac{M_1(x,N;1)}{x/N} = \mu(1,N) + O_\eps(N^{-\frac{171}{448}+\eps}) - \eta_{x, 1} NC_{1,\qt}\Big\{\log\Big(\frac x{\qt^2}\Big) + O(\log(2+N/\qt))\Big\}. $$
Since~$\qt, N\leq \e^{\delta\sqrt{\log x}}$ and~$\eta_{x,1}>0$, the last term here contributes a negative quantity for large enough~$x$, and we obtain the claimed inequality.
\end{proof}

\bibliographystyle{plain}
\bibliography{moment1-siegel}

\end{document}